\documentclass[reqno]{amsart}
\usepackage{amsfonts, amsmath, amsthm, amssymb, latexsym, graphicx}

\theoremstyle{plain}
\newtheorem{theorem}{Theorem}[section]
\newtheorem{corollary}[theorem]{Corollary}
\newtheorem{lemma}[theorem]{Lemma}
\newtheorem{proposition}[theorem]{Proposition}
\theoremstyle{definition}

\theoremstyle{remark}

\newtheorem*{remarks}{Remarks}

\numberwithin{equation}{section}

\title[Olshanski spherical functions for   motion groups]{ Olshanski spherical 
functions for infinite dimensional  motion groups of fixed rank}
\author{Margit R\"osler} 
\address{Institut f\"ur Mathematik, Universit\"at Paderborn, Warburger Str. 100,
33098 Paderborn, Germany}
\email{roesler@math.upb.de}
\author{Michael Voit}
\address{Fakult\"at Mathematik, Technische Universit\"at Dortmund,
          Vogelpothsweg 87,
          D-44221 Dortmund, Germany}
\email{michael.voit@math.tu-dortmund.de}


\begin{document}

\begin{abstract}
Consider the Gelfand pairs $(G_p,K_p):=(M_{p,q} \rtimes U_p,U_p)$ associated
with  motion groups
over  the fields $\mathbb F=\mathbb R,\mathbb C,\mathbb H$ with
$p\geq q$ and fixed $q$
as well as the inductive limit  $p\to\infty$,
 the Olshanski spherical pair $(G_\infty,K_\infty)$. We
classify
all
 Olshanski spherical functions  of $(G_\infty,K_\infty)$ as functions on  the
cone $\Pi_q$
of positive semidefinite $q\times q$-matrices
and show that they
 appear as (locally) uniform  limits of  spherical functions of $(G_p,K_p)$ as
$p\to\infty$. The latter are given by Bessel functions on $\Pi_q$.
Moreover, we determine all positive definite Olshanski spherical functions
and discuss related positive integral representations
 for matrix Bessel functions.
We also extend the results to the pairs
$(M_{p,q} \rtimes (U_p\times U_q),(U_p\times U_q))$ which are related to the
Cartan motion groups of non-compact Grassmannians. Here
Dunkl-Bessel functions of type B (for finite $p$)  and of
type A (for $p\to\infty$) appear as spherical functions.
\end{abstract}

\maketitle

\section{Introduction}
Let $M_{p,q}:=M_{p,q}(\mathbb F)$ be the vector space of  $p\times
q$-matrices over one of the fields
 $\mathbb F=\mathbb R,\mathbb C,\mathbb H$ and $U_p = U_p(\mathbb F)$ the group
of unitary $p\times p$-matrices over $\mathbb F$.
Consider the Gelfand pairs $(G_p,K_p):=(M_{p,q} \rtimes U_p,U_p)$ associated
with motion groups
over $\mathbb F$ where
  $q$ is fixed and $p\ge q$ is increasing.
The inductive limit
$(G_\infty,K_\infty):=(  \lim_{\rightarrow} G_p ,\lim_{\rightarrow} K_p) =
(M_{\infty, q}\rtimes U_\infty, U_\infty)$
  is an  Olshanski spherical pair;
 see \cite{F}, \cite{Ol1}, \cite{Ol2} for the notion.
The main purpose of this paper is to classify all Olshanski spherical
functions of this pair, i.e.,
 the  continuous $K_{\infty}$-biinvariant functions $\phi:G_\infty\to \mathbb C$
 which satisfy
 the product
formula
\begin{equation}\label{def-ol-sph-intro}
\phi(g)\phi(h)\,=\lim_{p\to\infty}\int_{K_p} \phi(gkh)\>
dk\quad\quad\text{for}\quad g,h\in G_\infty
\end{equation}
with respect to the normalized Haar measures $dk$ on $K_p$. We shall obtain a
 classification in terms of certain exponential functions which relies on the
fact that
 for integers $p\geq q$ as well for $p=\infty$, the double coset space
$G_p//K_p$ may
be identified with
 the cone  $\Pi_q$ of positive semidefinite $q\times q$-matrices over $\mathbb
F$,  and that for
finite $p$,
 formula (\ref{def-ol-sph-intro})  can be made explicit by the
results of
\cite{R1}.
It is  known (see e.g.~\cite{H}, \cite{FK}, \cite{R1}
 and references cited there) that for finite $p$, Bessel functions $J_\mu$ of
index $\mu=pd/2$ on the cone
  $\Pi_q$ provide
  spherical functions of $(G_p, K_p);$ here $d:=dim_{\mathbb
R}\mathbb F=1,2,4$. Moreover, it is also known that
under suitable rescaling, the Bessel functions $J_\mu$  converge (locally)
 uniformly to  exponential functions, see \cite{Ol1} and \cite{RV2}, and
references therein. From this we obtain that
all Olshanski spherical functions appear as
(locally) uniform
 limits of  spherical functions
 on $(G_p,K_p)$ for  $p\to\infty$.

 The classification of those
 Olshanski spherical functions which are positive definite is also easily
achieved.
The connection between positive definite spherical
functions of $(G_p, K_p)$ with different values of $p$  leads to
the existence of  positive integral representations for the involved matrix
Bessel
functions $J_{pd/2}$ as already explained by Schoenberg \cite{S} in a general
setting.
 Following mainly \cite{H} and \cite{FK} we shall
derive  explicit formulas for these  integral representations of
$J_\nu$ in terms of $J_\mu$
for $\nu\ge\mu$   in two different ways. A comparison of both appoaches
will imply an identity for projections of beta distributions on matrix cones.

Besides the Gelfand pairs $(M_{p,q}\rtimes U_p, U_p)$ we also consider the
Cartan motion
pairs
$(M_{p,q} \rtimes (U_p\times U_q),U_p\times U_q)$  and the associated
Olshanski
 spherical functions as $p\to\infty$. For finite $p$, the
spherical functions are now
  Dunkl-Bessel functions of type B which converge for $p\to\infty$ to certain
Dunkl-Bessel functions of type A, and these in turn are just the Olshanski
spherical functions
of $(M_{\infty,q} \rtimes (U_\infty\times U_q), U_\infty\times U_q)$; see
\cite{BF}, \cite{Du}, \cite{O}, \cite{R1}, \cite{RV1} for the background.

The approach  to  Olshanski spherical functions taken in this paper is similar
to the one for
noncompact infinite dimensional Grassmannian manifolds in \cite{RKV} where the
finite dimensional
spherical functions are  Heckman-Opdam hypergeometric
 functions of type B which converge  to
hypergeometric
 functions of type A. In this case,
 the classification  of the Olshanksi spherical functions is
based on a product formula in \cite{R2}. The present
paper as well as  \cite{RKV} were partly motivated
by related results  in \cite{DOW}.

Our approach via special functions
  does not use explicit representation theory for
 Olshanski spherical  pairs
 in contrast to \cite{Ol1}, \cite{Ol2}, \cite{P} or further related
references.
This has the slight advantage that we need not restrict to positive definite
spherical functions first, but can deal with general spherical
functions from the beginning.

The paper is organized as follows: In Sections 2 and 3 we collect
 some facts about matrix Bessel functions and Dunkl-Bessel functions of
 type B and type A.
The Gelfand pairs
$(M_{p,q} \rtimes U_p,U_p)$ and the corresponding Olshanski spherical functions
will
 then be studied in Section 4,
while Section 5 is devoted to the pairs $(M_{p,q} \rtimes (U_p\times
U_q),U_p\times U_q)$.
In Section 6 we  return to  matrix Bessel functions  and study their positive
integral
 representations.

\section{Bessel functions  on matrix cones}

In this section we collect some  facts about Bessel functions on matrix
cones. The material is mainly taken from
 \cite{FK},  \cite{H},  and \cite{R1}, and is in part
slightly generalized.

Let $\mathbb F$ be one of the fields $ \mathbb R, \mathbb C$
or  $\mathbb H$ with real dimension $d=1,2$ or $4$ respectively.
Denote the usual conjugation in $\mathbb F$ by $t\mapsto \overline t$,
the real part  of $t\in \mathbb F$ by $\mathfrak R t = \frac{1}{2}(t + \overline
t)$.

For $p,q\in \mathbb N$ we denote  by $M_{p,q}:=M_{p,q}(\mathbb F)$
the vector space of all  $p\times q$-matrices
 over $\mathbb F$ and put $M_q:= M_q(\mathbb F):=   M_{q,q}(\mathbb F)$ for
abbreviation.
Let further
\[H_q = H_q(\mathbb F)= \{x\in M_{q}(\mathbb F): x= x^*\}\]
be the space of Hermitian $q\times q$-matrices over $\mathbb F$, where $\,x^* =
\overline x^t$ denotes the usual involution on $M_q$.
The  space $H_q$ is a real Euclidean
 vector space with scalar product
 $\langle x, y\rangle :=  {tr}(xy)$  and
 norm $\|x\|=\langle x, x\rangle^{1/2}$,   where  ${tr}$ is the trace. Notice
that
 ${tr}(xy)$ is real for $x,y\in H_q$. $H_q$ is a (Euclidean) Jordan algebra
with the usual Jordan product $x\circ y= \frac{1}{2}(xy+yx).$
The real dimension of $H_q$ is
$$n:=n(q,\mathbb F):=dim_{\mathbb R}H_q= q + \frac{d}{2}q(q-1).$$
We also need the complexification $H_q^{\mathbb C}$
of $H_q$ which is a Jordan algebra over $\mathbb C$ in the natural way. The
extension of the
scalar product $\langle\, .\,,\,.\,\rangle\,$  to a hermitian scalar product on
$H_q^{\mathbb C}$
will  again be denoted by $\langle x, y\rangle$.
For $\mathbb F=\mathbb R$ or $\mathbb C$,  the space $H_q^{\mathbb C}$ may be
written as
$$H_q^{\mathbb C}=\{a+ib:\> a,b\in H_q(\mathbb F)\}\subset M_q(\mathbb C).$$
 More precisely, for
$\mathbb F=\mathbb R$
we have $H_q^{\mathbb C}(\mathbb R)=\{a\in M_q(\mathbb C):\> a^t=a\}$, while for
$\mathbb F=\mathbb C$, $H_q^{\mathbb C}(\mathbb C)= M_q(\mathbb C)$;
c.f.~Section VIII.5 of \cite{FK}. The complexification of the Jordan algebra
$H_q(\mathbb H)$ can be desribed as follows, c.f.
Sections V.2 and  VIII.5 of \cite{FK}: One realizes matrices
$z\in H_q(\mathbb H)$ as complex Hermitian
matrices
$$z=\begin{pmatrix} x&y\\ -\bar y&\bar x\end{pmatrix}\in H_{2q}(\mathbb C)$$
where $x\in H_{q}(\mathbb C)$ is Hermitian and $y\in Skew_q(\mathbb C)$ is
skew-symmetric. Then with
$J:=\begin{pmatrix} 0&I_q\\ -I_q& 0\end{pmatrix}\in M_{2q}$, the mapping
$$z\mapsto -J z=\begin{pmatrix} \bar y&-\bar x \\ x & y\end{pmatrix}
\in Skew_{2q}(\mathbb C)$$
defines a real Jordan algebra isomorphism from $H_q(\mathbb H)$ onto the
Jordan algebra $V_q:=\{u\in Skew_{2q}(\mathbb C): \> u^*=JuJ\}$ with Jordan
product $\,u\circ v:=\frac12(uJv+vJu)$. The complexification of this
Jordan algebra is just $Skew_{2q}(\mathbb C)$ (with the same Jordan-product).
We thus identify
 $H_q^{\mathbb C}(\mathbb H)$ and $Skew_{2q}(\mathbb C)$ as complex Jordan
algebras.


Let
\[\Pi_q:=\{x^2:\> x\in H_q\} = \{ x^*x: x\in H_q\}\]
denote the set of all positive semidefinite
matrices in $H_q$, and $\Omega_q$ its
interior consisting of all strictly positive
definite matrices. $\Omega_q$ is a symmetric cone; see \cite{FK} for details.
On $ H_q$ we use the standard partial ordering
\[x\le y :\Longleftrightarrow \,
y-x\in\Pi_q\,.\]

To define  Bessel functions,  we need
the spherical polynomials
\[ \Phi_\lambda (x) = \int_{U_q} \Delta_\lambda(uxu^{-1})du, \quad x\in H_q\]
which are indexed
by partitions $\lambda = (\lambda_1 \geq \lambda_2\geq \ldots \geq \lambda_q)
\in \mathbb N_0^q$ (for short, $\lambda \geq 0$).
Here
$du$ denotes the normalized Haar measure of $U_q=U_q(\mathbb F)$, and $
\Delta_\lambda$  is the power function
\[ \Delta_\lambda(x) := \Delta_1(x)^{\lambda_1-\lambda_2}
\Delta_2(x)^{\lambda_2-\lambda_3} \cdot\ldots\cdot \Delta_q(x)^{\lambda_q}\quad
(x\in H_q)\]
with the  principal minors $\Delta_i(x)$ of $x$; see
\cite{FK}.
There is a renormalization
$Z_\lambda = c_\lambda \Phi_\lambda$
with suitable constants $c_\lambda >0$ depending on $\Pi_q$
such that
\begin{equation}\label{power-tr}
(tr \,x)^k \,=\, \sum_{|\lambda|=k} Z_\lambda(x)
\quad\quad\text{for}\quad
 k\in \mathbb N_0\,;
\end{equation}
see Section XI.5.~of \cite {FK} where the $Z_\lambda$ are called
zonal polynomials. By construction, the $Z_\lambda$ are invariant under
conjugation by $U_q$ and thus
 depend only on the eigenvalues of their argument.
More precisely, for $x\in H_q$ with eigenvalues $\xi = (\xi_1, \ldots, \xi_q)\in
\mathbb R^q$,
\begin{equation}\label{identjack}Z_\lambda(x) =
C_\lambda^\alpha(\xi) \quad \text{with}\quad \alpha = \frac{2}{d}\end{equation}
where the $C_\lambda^\alpha$ are the Jack polynomials of index
$\alpha$  (c.f. \cite {FK}, \cite{M}, \cite{R1}).
These are homogeneous of degree $|\lambda|$ and symmetric in
their arguments. Notice that the zonal polynomials $Z_\lambda$ naturally
extend to $H_q^\mathbb C$.

The Bessel functions associated with the cone $\Omega_q$ are defined
as  $_0F_1$-hyper\-geometric series
\begin{equation}\label{power-j}
 J_\mu(z) =
 \sum_{\lambda\geq 0} \frac{(-1)^{|\lambda|}}{(\mu)_\lambda|\lambda|!}
Z_\lambda(z),\quad z\in H_q^\mathbb C
\end{equation}
where
 the generalized Pochhammer symbol $(\mu)_\lambda$ is given by
\[ (\mu)_\lambda =\, (\mu)_\lambda ^{2/d} \quad \text{with }\,
(\mu)_\lambda^\alpha := \,\prod_{j=1}^q
\bigl(\mu-\frac{1}{\alpha}(j-1)\bigr)_{\lambda_j} \quad (\alpha >0). \]
The index $\mu\in \mathbb C$ is required to satisfy $(\mu)_\lambda^\alpha \not=
0$ for
all $\lambda \geq 0$. In this case, $ J_\mu$ is holomorphic on $ H_q^\mathbb C$.

 If $q=1,$ then $\Pi_q=\mathbb R_+$,  and the Bessel function $ J_\mu$ is
independent of $d$ with
\[ J_\mu\bigl(\frac{z^2}{4}\bigr) = j_{\mu-1}(z) \]
where $\, j_\kappa(z) = \,_0F_1(\kappa +1; -z^2/4)\,$
is the  modified Bessel function in one variable.

We need the following product formula for the $J_\mu$ from \cite{R1}:

\begin{proposition}\label{allg-prod-form}
For $\mu\in\mathbb C$ with $\Re\mu>d(q-1/2)$ and all $r,s\in H_q^{\mathbb
C}$,
$$ J_\mu(r^2)J_\mu(s^2)=\frac{1}{\kappa_\mu}\int_{B_q}
J_\mu(r^2+s^2+rws+sw^*r)\> \Delta(I_q-w^*w)^{\mu-\gamma}\> dw$$
where $\Delta$ denotes the determinant on $H_q$,  $dw$  the Lebesgue measure on
the ball
$$B_q:=\{w\in M_q:\> w^*w<I_q\},$$
$\gamma:=d(q-1/2)+1$, and $\kappa_\mu:=\left( \int_{B_q}
\Delta(I_q-w^*w)^{\mu-\gamma}\> dw  \right)^{-1}$.
\end{proposition}

\begin{proof} For  $r,s\in H_q$ we refer to Eq. (3.8) of \cite{R1}.
As both sides of the equation are holomorphic in  $r,s\in H_q^{\mathbb C}$, the
result is true in general.
\end{proof}

 For the limit case $\mu=d(q-1/2)$ there exists a degenerated product formula;
see Section 3.5 of \cite{R1}.

For real $\mu\ge d(q-1/2),$   product formula \ref{allg-prod-form}  leads
to a positive (hypergroup) convolution structure on the Banach space
$M_b(\Pi_q)$ of all bounded
signed Borel measures
on $\Pi_q$, as follows: define first the convolution of point measures
\begin{equation}\label{def-convo}
(\delta_r *_\mu \delta_s)(f) :=
\frac{1}{\kappa_\mu}\int_{B_q} f\bigl(\sqrt{r^2 + s^2 + swr + rw^*\!s}\,\bigr)\,
\Delta(I-ww^*)^{\mu-\rho}\, dw \quad(r,s\in\Pi_q)
\end{equation}
for $f\in C_b(\Pi_q)$, and then extend  this convolution in a weakly
continuous,
bilinear way to  $M_b(\Pi_q)$.
This generates a probability preserving commutative Banach algebra
$(M_b(\Pi_q),*_\mu)$;
see \cite{R1} for details.
For  $\mu=pd/2$ with integers $p\ge 2q$, this convolution is just the double
coset convolution associated with the Gelfand pair
 $( M_{p,q} \rtimes U_p,U_p)$ (see  \cite{R1}) which we shall consider
in more detail in Section 4.

For $s\in H_q^{\mathbb C}$ we  define
 the continuous function
\begin{equation}\label{def-f-s}
f_s^\mu(r):=  J_\mu(\frac{1}{4}rsr) \quad\quad (r\in H_q^{\mathbb C}).
\end{equation}
as well as the function 
\begin{equation}\label{def-phi-s}
\phi_s^\mu(r):=f_{s^2}^\mu(r) = J_\mu(\frac{1}{4}rs^2r),
\end{equation}
where the latter
definition is in 
correspondence with the notion in \cite{R1}.
The following facts slightly generalize  results from \cite{R1}.

\begin{lemma}\label{properties-phi}
\begin{enumerate}\itemsep=-1pt
\item[\rm{(1)}] For all $r,s\in  H_q^{\mathbb C}$,
$\phi_s^\mu(r)=\phi_r^\mu(s)$.
\item[\rm{(2)}] For all $r,t\in\Pi_q$ and  $s\in H_q^{\mathbb C}$,
\begin{equation}\label{phi-s-prod-form}
f_s^\mu(r)f_s^\mu(t)=\int_{\Pi_q} f_s^\mu(z)\>
d(\delta_r*_\mu\delta_t)(z),
\end{equation}
i.e.,  the functions $f_s^\mu$ with  $s\in H_q^{\mathbb C}$ are
multiplicative with respect to
$*_\mu$.
\end{enumerate}
\end{lemma}

\begin{proof} Both statements are known for $r, s, t\in \Pi_q$; see the results
for $\phi_s$ in
Section 3 of  \cite{R1}
 and notice for the second statement that each $s\in\Pi_q$ may be written as
$s=\tilde s^2$ with some unique
$\tilde s\in\Pi_q$. The general statements now follow easily by analytic
continuation from  $\Pi_q$
 to $H_q^\mathbb C$. For this, notice that $H_q^{\mathbb C}\setminus \Pi_q$ is
connected, and that the identity theorem can be applied because $\Pi_q$ has
non-empty interior in $H_q$.
\end{proof}

\section{Dunkl-Bessel functions on Weyl chambers of type B and type
A}\label{Dunkl-Bessel}
There is a close connection between the Bessel functions on  the cone $\Pi_q$
and the theory of Dunkl operators associated with the root system  $B_q$, see
\cite{R1}.
 We briefly  review this connection. We do not go into details of Dunkl
 theory, but refer to \cite{BF}, \cite{Du}, \cite{O}
and \cite{R1}.    For a reduced root system $R\subset \mathbb R^q$ and a
multiplicity
function $k:R\to [0,\infty)$  (i.e. $k$ is  invariant under the
action of the corresponding reflection group),
we denote by $ J_k^R$ the Dunkl-Bessel function associated with $R$ and $k$.
 It is obtained from the Dunkl kernel by symmetrization
 with respect to the underlying reflection group. Dunkl-Bessel functions
generalize
 the spherical functions of Euclidean type symmetric
spaces, which occur for crystallographic root systems and specific discrete
values of $k$, see \cite{O}. For the root system $A_{q-1}= \{\pm(e_i-e_j):\,
i<j\}\subset \mathbb R^q$, the
multiplicity $k$ is a single real parameter. If $k>0,$ then according to
formulas (3.22) and (3.37)  of \cite{BF} the
associated Dunkl-Bessel function is
 the  generalized $_0F_0$-hypergeometric function
\[ J_k^A(\xi,\eta) =\, _0F_0^\alpha (\xi,\eta) := \sum_{\lambda \geq 0}
\frac{1}{|\lambda|!}\cdot
\frac{C_\lambda^\alpha(\xi) C_\lambda^\alpha(\eta)}{C_\lambda^\alpha({\bf
1})}\quad (\xi, \, \eta \in \mathbb C^q),\]
with $\, {\bf 1} = (1,\dots,
1),\, \alpha = 1/k.$
For $k=\frac{d}{2}$ with $d= \text{dim}_\mathbb R \mathbb F,$ the
Dunkl-Bessel functions $\xi\mapsto J_{d/2}^A(\xi, \eta)$ are
known to be the spherical functions of the flat symmetric space
$H_q(\mathbb F)\rtimes U_q(\mathbb F)/U_q(\mathbb F)$ and therefore have the
Harish-Chandra
type integral representation
\begin{equation}\label{int-rep-a-fkt}
 J_{d/2}^A(\xi,\eta) = \int_{U_q(\mathbb F)} e^{\text{tr}(\eta u \xi u^{-1})}
du.\end{equation}
See also  formula (7) of \cite{RV1}, where an alterantive proof of
\eqref{int-rep-a-fkt}
is
given.

For the root system   $B_q = \{\pm e_i , \,\pm e_i \pm e_j: \, i <j \}$,
we have  $k=(k_1, k_2)$ where $k_1$ and $k_2$ correspond to
 the roots $\pm e_i$ and  $\pm e_i \pm e_j$ respectively. The associated
Dunkl-Bessel function is
\[ J_k^B(\xi,\eta) = \, _0F_1^\alpha\bigl(\mu; \frac{\xi^2}{2},
\frac{\eta^2}{2}\bigr)
\quad \text{with }\, \alpha = 1/k_2, \, \mu = k_1 +(q-1)k_2 +1/2\]
where $\xi^2= (\xi_1^2, \ldots, \xi_q^2)$ and
\[_0F_1^\alpha(\mu; \xi, \eta) := \, \sum_{\lambda\geq 0}
 \frac{1}{(\mu)_\lambda^\alpha |\lambda|!}\cdot \frac{C_\lambda^\alpha(\xi)
 C_\lambda^\alpha(\eta)}{C_\lambda^\alpha({\bf 1})}.\]
The Dunkl-Bessel function $J_k^B$  is invariant in both arguments under the
action of the hyperoctahedral group which is generated by sign changes and
permutations of the coordinates.
For suitable multiplicity $k$, it is
related the matrix
Bessel function $J_\mu$ and to the functions
 $\phi_s^\mu(r)=  J_\mu\bigl(\frac{1}{4}sr^2s\bigr)$ of
Section 2, as follows:
Consider the  $B_q$-Weyl chamber
\begin{equation}\label{Weyl_B} C_q^B := \{ \xi= (\xi_1, \ldots, \xi_q) \in
\mathbb R^q: \, \xi_1 \geq \ldots \geq \xi_q \geq 0\}\end{equation}
For $\xi\in C_q^B$ we denote by
$\underline \xi\in \Pi_q$ the  diagonal matrix with entries $\,\xi_1, \ldots,
\xi_q\,.$  Let
\[k(\mu, d) := \bigl(\mu - (d(q-1)+1)/2, \,d/2\bigr).\]
 Then for $\xi,\eta\in
C_q^B,$
\begin{equation}\label{Dunklchar}
 J_{k(\mu,d)}^B(\xi, i\eta)= \int_{U_q} J_\mu\bigl(\frac{1}{4}\underline\eta
 u\underline\xi^2
 u^{-1}\underline\eta\bigr) du;
\end{equation}
see Section 4 of \cite{R1}.
Dunkl-Bessel functions of type $B$ will play an important role in Section 5
of this paper, in connection with the study of the Olshanski
  spherical pairs $(M_{\infty,q}\rtimes (U_\infty\times U_q), U_\infty\times
U_q).$

\section{Olshanski spherical functions related to $M_{\infty,q}\rtimes
U_\infty$}

In this section, we consider Olshanski spherical pairs associated with
matrix cones.
For a general background on  Olshanski spherical pairs and their spherical
functions we refer to Faraut
\cite{F} and
Olshanski \cite{Ol1}, \cite{Ol2}.

 We
fix the field $\mathbb F\in \mathbb \{\mathbb R, \mathbb C, \mathbb H\}$ and
the rank $q\in\mathbb N$ as in Section 2.
Consider the Gelfand pairs $(G_p,K_p)$ with $G_p= M_{p,q} \rtimes U_p$
and $K_p:=U_p$ where $U_p$ acts on $M_{p,q} $ by left multiplication.
 In all  cases, $G_p$
 will be regarded as a closed subgroup of $G_{p+1}$ with $K_p=G_p\cap K_{p+1}$.
 Consider the inductive limits 
 $G_\infty := \lim_{\rightarrow} G_p$ 
and $K_\infty:= \lim_{\rightarrow} K_p$.  Then $(G_\infty,
 K_\infty)$ is an Olshanski spherical pair.
A continuous function $\phi:G_\infty\to \mathbb C$  is called an Olshanski
spherical 
function of $(G_\infty,K_\infty)$  if $\phi$ is $K_{\infty}$-biinvariant and
satisfies the product
formula
\begin{equation}\label{def-ol-sph}
\phi(g)\phi(h)\,=\lim_{p\to\infty}\int_{K_p} \phi(gkh)\>
dk\quad\quad\text{for all } g,h\in G_\infty
\end{equation}
with respect to the normalized Haar measure $dk$.
Further, an Olshanski spherical
function $\phi$ of $(G_\infty,K_\infty)$ is called positive definite 
if  $\phi$ is positive definite on $G_\infty$ in the usual sense, i.e. if for
all $n\in\mathbb N$, $g_1,\ldots,g_n\in G_\infty$ and $c_1,\ldots,c_n\in\mathbb
C$,
$$\sum_{k,l=1}^n c_k \overline{c_l} \,\phi(g_kg_l^{-1})\ge0.$$
Positive definite continuous functions on $G_\infty$ are always bounded.
Furthermore, in our situation
 $$K_\infty g^{-1}K_\infty=K_\infty g K_\infty \quad\quad\text{for} \quad\quad
g\in G_\infty.$$
This implies that positive definite spherical functions on $G_\infty$ are
automatically $\mathbb
R$-valued.

We shall now classify the Olshanski spherical 
functions of $(G_\infty, K_\infty)$ as well as those which are positive definite
in addition.

For this we recapitulate that for each $p$, the space of
double cosets $G_p/\!/K_p$ can be topologically identified with the cone $\Pi_q$
via the homeomorphism
$$ K_p (x,k)K_p  \,\mapsto \,\sqrt{x^*x} \quad\quad \text{for } \,
x\in M_{p,q}, \,k\in K_p\,.$$
 In other words, if for $a\in\Pi_q$ we consider the matrix $\begin{pmatrix}a\\
0\end{pmatrix}\in
M_{p,q}$,  then
the inverse of the above homeomorphism can be written as
$$F_p:\Pi_q\to \,G_p/\!/K_p,\quad r\longmapsto K_pg_a K_p
\quad\quad \text{with} \quad
g_a:=(\begin{pmatrix}a\\ 0\end{pmatrix},I_p).$$ 
By definition of the inductive limit topology, the mapping
$$F:\Pi_q\to \,G_\infty/\!/K_\infty,\quad a\longmapsto K_\infty g_a K_\infty
\quad\quad \text{with} \quad
g_a:=(\begin{pmatrix}a\\ 0_\infty\end{pmatrix},I_\infty)$$ 
also provides a homeomorphism for $p=\infty$.
We will therefore use the agreement that for all integers $p\geq q$ as well as
for
$p=\infty$, 
 continuous and
$K_p$-biinvariant functions on
$G_p$ are identified with   continuous
 functions on $\Pi_q$. When doing so, our notations immediately imply the
following

\begin{lemma}\label{posdef-p-infty}
Let $p_1\ge p_2\ge q$. If a continuous function on $\Pi_q$ corresponds to a
$K_{p_1}$-biinvariant, positive definite  function on $G_{p_1}$, then it also
corresponds to a
$K_{p_2}$-biinvariant, positive definite function on $G_{p_2}$.
Moreover, a continuous function on $\Pi_q$ corresponds to a
$K_{\infty}$-biinvariant, positive definite  function on $G_{\infty}$
 if and only if there exists an integer $p_0\ge1$ such that for 
all integers $p\ge p_0$, it corresponds to a
$K_{p}$-biinvariant, positive definite function on $G_{p}$.
\end{lemma}

\begin{proof} The first  statement follows immediately from the definition of
  positive definite functions and the fact that for  $p_1\ge p_2\ge q$
(where possibly
  $p_1=\infty$), the canonical projections
\[P_p:G_p\to\,
  G_p/\!/K_p\equiv\Pi_q\,,\quad P_p((x,k)):=\sqrt{x^*x}\]
 satisfy
$P_{p_1}|_{G_{p_2}}=P_{p_2}$.
The remaining part of the second statement  follows also easily by these
arguments.
\end{proof}

We next turn to the (not necessarily positive definite) spherical functions
 of $(G_p, K_p)$ for finite $p\geq q.$ 
We know from
Lemma \ref{properties-phi} that the functions $f_s^{pd/2}$  of Section 2 
with $s\in H_q^{\mathbb C}$  are spherical. On the other hand, using results
 of Wolf \cite{W1} in combination
with an integral representation of the matrix Bessel functions  $f_s^{pd/2}$ 
 (see Eq. (3.4) of  \cite{R1} and Propos. XVI.2.3 of \cite{FK}), we obtain the
following classification:

\begin{theorem}\label{class-spher-1}
Let $p\geq q$ be finite. Then $\{f_s^{pd/2}: \> s\in
H_q^{\mathbb C}\}$ is the set of all  spherical functions
of $(G_p,K_p)$.
Moreover, the  set of positive definite  spherical functions of
$(G_p,K_p)$ is given by  $\{f_{s}^{pd/2}: s\in \Pi_q\}.$
\end{theorem}

\begin{proof}
According to \cite{W1} (see also \cite{W2}), 
 the complete set of
spherical functions of $(G_p, K_p)$ can be described as follows: 

Consider $M_{p,q}$ as a real vector space  of dimension $dpq$ with Euclidean
scalar product
\[(x|y):=\Re tr(x^*y)\]
 and extend this form in a bilinear way to the
complexification  $M_{p,q}^{\mathbb C}$
of  $M_{p,q}$. Then it is easily checked that for each $y\in M_{p,q}^{\mathbb
C}$, the
function \begin{equation}\label{SF_Wolf}
\tilde\phi_y(x,v):= \int_{U_p} e^{-i(ux|y)}\> du \quad\quad (x\in M_{p,q},
v\in U_p)\end{equation}
defines a spherical function of  $(G_p,K_p)$.
Moreover, by Theorem 4.4 of
\cite{W1},
 all spherical functions are given in this way, and by Theorem 5.4 of
\cite{W1}, the set of positive-definite
spherical functions is made up by those $\tilde\phi_{y}$  with $y\in
M_{p,q}$.

We next show that all functions of the form \eqref{SF_Wolf} are in fact 
 Bessel functions $f_s^{pd/2}$ for suitable $s\in
H_q^{\mathbb C}$. For this we again regard $M_{p,q}$ and $H_q$ as vector spaces
over $\mathbb R$, and 
denote the complex-bilinear extension of the $\mathbb R$-bilinear mapping
$$M_{p,q}\times M_{p,q}\to M_q,\quad (x,y)\longmapsto x^*y$$
(to a mapping $M_{p,q}^{\mathbb C}\times M_{p,q}^{\mathbb C}\to M_q^{\mathbb
C}$) again by $ x^*y$. For 
$y\in M_{p,q}^{\mathbb C}$ we then obtain easily that $y^*y\in  H_q^{\mathbb
C}$.

Now fix a matrix $y\in  M_{p,q}^{\mathbb C}$. Then $y^*y\in  H_q^{\mathbb C}$,
and for all $r\in \Pi_q$,
$$f_{y^*y}^{pd/2}(r)=J_{pd/2}(\frac{1}{4}ry^*yr)=J_{pd/2}(\frac{1}{4}
(yr)^*yr).$$
We conclude from Eqs. (3.3) and (3.4) of  \cite{R1} (see also Propos. XVI.2.3 of
\cite{FK})
that for all $x\in M_{p,q}$,
\begin{equation}\label{ident-bessel-int}
J_{pd/2}(\frac{1}{4}x^*x)=
\int_{U_p} e^{-i (u\sigma_0 |x)} du \quad \text{ with}\, \,\sigma_0 =
\begin{pmatrix} I_q\\ 0\end{pmatrix}\in M_{p,q}.
\end{equation}
Using the  complex-bilinear and thus analytic extension above, we conclude that 
(\ref{ident-bessel-int}) remains correct for all  $x\in M_{p,q}^{\mathbb C}$.
Now let $r\in\Pi_q$, $y\in M_{p,q}^{\mathbb C}$. Then 
$$J_{pd/2}(\frac{1}{4}(yr)^*yr)=
\int_{U_p} e^{-i (u\sigma_0 |yr)} du.$$
As
$$ (u\sigma_0 |yr)=\Re tr((u\sigma_0 )^* yr)=\Re tr((u\sigma_0r)^*y)=(u\sigma_0r
|y),$$
we obtain
\begin{equation}\label{ident-bessel-wolf}
f_{y^*y}^{pd/2}(r)=\int_{U_p} e^{-i (u\sigma_0r |y)} du=
\tilde\phi_y(\sigma_0 r,v)
\end{equation}
 with arbitrary $v\in U_p\,.$ As the functions on both sides are biinvariant,
and the  $r\in\Pi_q$
form a set of representatives of all double cosets as descibed in the beginning
of this section, 
the proof of the first statement of the theorem is complete.
Moreover, equation (\ref{ident-bessel-wolf}) in combination with Theorem 5.4 of
\cite{W1} leads to the stated classification of the positive definite spherical
functions.
\end{proof}

We mention that the statement about the positive definite  spherical functions
above
can be also obtained by
 hypergroup methods
from Theorem 3.12 of \cite{R1} in combination with results of \cite{J}.

\medskip

We now turn to the case $p=\infty$.
The
 Olshanski spherical 
functions of $(G_\infty, K_\infty)$ can be characterized as follows:

\begin{lemma}\label{limit-product-formula}
A continuous  $K_\infty$-biinvariant function $\phi:G_\infty\to \mathbb C$ 
is  Olshanski spherical 
if and only if the continuous function $\tilde\phi(b):=\phi(g_b)$ on $\Pi_q$
 satisfies  the product
formula
\begin{equation}\label{prod-formel-1}
\tilde\phi(a)\cdot \tilde\phi(b)
\,=\,\tilde\phi(\sqrt{a^2+b^2}), 
 \quad a, b\in \Pi_q.
\end{equation}
\end{lemma}

\begin{proof} Let $\phi$ be a continuous  $K_\infty$-biinvariant function on
  $G_\infty$ and $\tilde\phi\in C(\Pi_q)$ as described in the lemma.
  Then, by (\ref{def-ol-sph}) and the product formula (\ref{def-convo}),
 $\phi$ is  Olshanski spherical iff $\tilde\phi$
  satisfies
\begin{equation}\label{limit-eq1}
\tilde\phi(a)\cdot \tilde\phi(b)
\,=\,\lim_{p\to\infty} \frac{1}{\kappa_{pd/2}}
\int_{B_q}\tilde\phi(\sqrt{a^2+b^2+awb+bw^*a}) \cdot
\Delta(I-w^*w)^{pd/2-\gamma}\> dw
\end{equation}
for $a,b\in  \Pi_q$ with $\gamma=d(q-1/2)+1$.
The probability measures 
$$\,\kappa_{pd/2}^{-1}\cdot\Delta(I-w^*w)^{pd/2-\gamma} \,
dw\,$$
 are compactly supported in $B_q$ and tend weakly to the point
measure $\delta_0$ for $p\to \infty$. Therefore (\ref{limit-eq1}) is equivalent
to
 \[
\tilde\phi(a)\cdot \tilde\phi(b)
\,=\,\tilde\phi(\sqrt{a^2+b^2})
\]
as claimed.
\end{proof}

We remark that precise estimates for the order of convergence of the
 probability measures $\,\kappa_{pd/2}^{-1}\cdot\Delta(I-w^*w)^{pd/2-\gamma} \,
dw\,$ are given in \cite{V2}.

We now solve the functional equation (\ref{prod-formel-1}):

\begin{lemma}\label{sol-conv-eq1}
A continuous function  $\tilde\phi\in C(\Pi_q)$ with $\tilde\phi(0)=1$
satisfies 
  (\ref{prod-formel-1}) if and only if there exists some  $b\in H_q^{\mathbb
  C}$ such that
\[\tilde\phi(a)=\, \exp\bigl(-\langle a^2,b\rangle\bigr) \,=:\psi_b(a), \quad
a\in \Pi_q\,.\]
\end{lemma}

\begin{proof}
Clearly, all $\psi_b$ satisfy  (\ref{prod-formel-1}).
Conversely, if a function  $\tilde\phi\in C(\Pi_q)$  satisfies
  (\ref{prod-formel-1}), then $\psi(a):=\tilde\phi(\sqrt a)$ satisfies 
$\psi(a)\psi(b)=\psi(a+b)$ for $a,b\in\Pi_q$ and thus, because of 
  $\psi(0)=1$ and continuity, $\psi(a)\ne0$ for all  $a\in\Pi_q$.
Each  $a\in H_q$ can be written as $a=d-cI_q$ for some $d\in \Pi_q     $ and
some 
$c\in[0,\infty[$, and it is easily checked that
 $\psi(a):=\psi(d)/\psi(cI_q)$
 defines a well-defined function $\psi\in C(H_q      )$ with
$\psi(a)\psi(b)=\psi(a+b)$ for $a,b\in H_q^{\mathbb C}      $.
 The assertion now follows by the  well-known
 characterization of the exponential function on a Euclidean space 
by its functional equation. 
\end{proof}

The preceding lemmata immediately lead  to the following characterization of the
 Olshanski spherical  functions of $(G_\infty, K_\infty)$.

\begin{theorem}\label{ident-ol}
A continuous  $K_\infty$-biinvariant function $\phi: G_\infty\to \mathbb C$ is
Olshanski spherical 
 if and only if 
$$\phi(g_a)=\psi_b(a)=\exp\left(-  \langle a^2,b\rangle \right) \quad\quad
(a\in\Pi_q      ) $$
 for some  $b\in H_q^{\mathbb C}      $.
\end{theorem}

We next investigate which  Olshanski spherical functions of
$(G_\infty,K_\infty)$
appear as limits of spherical functions on $(G_p,K_p)$ for $p\to\infty$.
For this we employ the known convergence of $J_\mu(\mu y)$ to  $e^{-tr(y)}$ for
$\mu\to\infty$:


\begin{lemma}\label{konv-matrix-bessel}
For $y\in  H_q^{\mathbb C}$, \[\lim_{\mu\to\infty}J_\mu(\mu
y)=e^{-tr(y)}\quad\text{ locally uniformly.}\]
On the cone $\Pi_q$, this convergence is even uniform.
More precisely, there
exists a constant $C=C(q,\mathbb F)$ such that
$$|J_\mu(\mu y)-e^{-tr(y)}|\le \,C/\mu
\quad\quad\text{for all }\, y\in\Pi_q      ,\> \mu\ge 2q.$$
\end{lemma}

\begin{proof}
For the second statement we  refer to Proposition 3.2 of  \cite{RV2}; see also 
Lemma 4.8 of \cite{Ol1} or references cited there for the uniform convergence.

The first statement can be obtained either  from the  integral representation
(3.12) of \cite{R1} for $J_\mu$, or 
 by power series expansion as in the proof of 
Proposition 3.5 in \cite{RV2}. We outline the second approach:
Using the expansions (\ref{power-j}) and (\ref{power-tr}),  we have
$$J_\mu(\mu y)-e^{-tr(y)}  = \sum_{\lambda\geq 0} 
\frac{1}{|\lambda|!}\Bigl( \frac{\mu^{|\lambda|}}{(\mu)_\lambda}-1\Bigr)
\cdot Z_\lambda(-y)$$
where by Lemma 3.4 of \cite{RV2}
$$\Bigl|1-\frac{\mu^{|\lambda|}}{(\mu)_\lambda}\Bigr|\le 
dq \cdot 2^{dq(q-1)/2}\cdot\frac{ |\lambda|^2}{\mu}.$$
It is easily checked that the series
$$ \sum_{\lambda\geq 0}\frac{ |\lambda|^2}{|\lambda|!}\cdot Z_\lambda(-y)$$
converges absolutely and locally uniformly for $y\in H_q^{\mathbb C}$
(c.f. Section XV.1 of \cite{FK} and our normalization of the  $Z_\lambda$
instead of the $\Phi_\lambda$
there). This immediately leads to the locally uniform convergence of order
$1/\mu$.
\end{proof}

We conclude from Lemma \ref{konv-matrix-bessel} that for all $y\in\Pi_q$, $b\in
H_q^{\mathbb C}$
and for $\mu=pd/2\to\infty$, the functions $f_s^\mu$ of Section 2 satisfy
$$f_{\mu b}^\mu(y)=J_\mu(\frac{\mu}{4}yby)\to
exp(-tr(yby)/4)=\psi_{b/4}(y)$$
uniformly or locally uniformly depending on  $b$. According to Theorem
\ref{class-spher-1}, the functions
$f_{\mu b}^\mu$  with
$b\in H_q^{\mathbb C}$ form the spherical functions of $(G_p,K_p)$. Considering
biinvariant functions on 
$G_p$ and $G_\infty$ as functions on
 $\Pi_q$ as above, we obtain

\begin{corollary}\label{limit-spher-c1}
All Olshanski spherical functions $\psi_{b}$ of $(G_\infty,K_\infty)$ with $
b\in H_q^{\mathbb C}$ appear as
locally uniform limits of  spherical functions of  $(G_p,K_p)$. 

Moreover,
those  Olshanski spherical functions $\psi_b$  with $b\in \Pi_q $
 even appear as uniform limits of positive definite  spherical functions of
$(G_p,K_p)$ as
$p\to\infty$. 
\end{corollary}

We finally determine the  Olshanski spherical functions $\psi_b$ which are
positive definite.

\begin{theorem}\label{pos-def1}
The positive definite Olshanski spherical functions of $(G_\infty,K_\infty)$
are precisely given by the functions  $\psi_b$  with $b\in \Pi_q      $.
\end{theorem}

\begin{proof}
Assume that  $\psi_b$ is a positive definite Olshanski spherical function. Then
$\psi_b$
must be $\mathbb R$-valued. From $\psi_b(a)=exp\left(-  \langle a^2,b\rangle
\right)$ for $a\in \Pi_q$ we infer that  $ b\in H_q$.
 Moreover, as $\psi_b$ is in addition  bounded on $\Pi_q$, it follows easily
that  $ b\in \Pi_q $.
In fact, if $ b$ would have a negative eigenvalue with eigenvector $u$,
 we could choose a 
matrix $a\in \Pi_q $ with the same eigenvector $u$ associated to the eigenvalue
1,
and with 
all other eigenvalues equal to 0. It is then clear that 
  $\psi_b(ca)=exp\left(- c^2 \langle a^2,b\rangle
\right)$ tends to
 $\infty$ for $c\to\infty$.

For the converse statement consider  $b\in \Pi_q$.
Let $p\ge q$ be a fixed integer. Then for each integer  $\tilde p\ge p$, the
spherical function $\,f_b^{\tilde pd/2}\in C(\Pi_q)$ corresponds to a $K_{\tilde
p}\,$-biinvariant
positive definite function on $G_{\tilde p}$, and thus by Lemma
\ref{posdef-p-infty}, to a $K_{ p}$-biinvariant
positive definite function on $G_{ p}$.
As positive definiteness is preserved under limits,
it follows from Corollary \ref{limit-spher-c1} that  $\psi_b\in C(\Pi_q)$
is a
$K_p$-biinvariant  positive definite function on $G_{ p}$. This holds for
all $p$, and therefore Lemma \ref{posdef-p-infty} ensures that $\psi_b$
is a
positive definite function on $G_\infty$ as claimed.
\end{proof}

In summary, the Olshanski spherical functions of $(G_\infty,K_\infty)$ admit a
classification which is in complete accordance with that for finite $p$
in Theorem \ref{class-spher-1}.

\section{Olshanski spherical functions related to 
$M_{\infty,q}\rtimes(U_\infty\times U_q)$}

In this section we consider the Gelfand pairs $(G_p,K_p)$ with
 $G_p= M_{p,q}       \rtimes (U_p\times U_q)$
and $K_p:=U_p\times U_q$ for fixed $q\ge1$, where the
groups $U_p$ and  $U_q$ act on $M_{p,q}      $ by multiplication
from the left and right, respectively. Consider  the Olshanski spherical pairs
 $(G_\infty,K_\infty)$ with $G_\infty := \lim_{\rightarrow} G_p$ 
and $K_\infty:= \lim_{\rightarrow} K_p$.

Again, we investigate the 
 Olshanski spherical 
functions of $(G_\infty,K_\infty)$.
We recapitulate first that for each $p,$ the double coset space $G_p/\!/K_p$
 can be topologically identified with the Weyl chamber
$$C_q^B:=\{\xi=(\xi_1,\ldots,\xi_q)\in \mathbb R^q:\> \xi_1\ge \ldots\ge
\xi_q\ge0\}$$
of type $B$ via
$$K_p(x,k)K_p\mapsto\sigma(\sqrt{x^*x}\,) 
\quad\quad\text{for} \quad\quad
x\in M_{p,q}      ,\, k\in K_p$$
independently of $p$ where $\sigma$ stands for the ordered spectrum of a
positive
semidefinite matrix.
 In other words, if for $\xi\in C_q^B$  we consider the diagonal matrix
$\underline \xi:=diag(\xi_1, \ldots, \xi_q)\in M_q      $ as well as
$\begin{pmatrix}
\underline \xi
\\ 0\end{pmatrix}\in M_{p,q}      $,  then
the inverse of this homeomorphism can be written as
$$F_p:C_q^B\to G_p/\!/K_p,\quad b\longmapsto K_pg_\xi K_p 
\quad\quad \text{with} \quad\quad
g_\xi:=\bigl(\begin{pmatrix}  \underline \xi   \\ 0\end{pmatrix},I_p\times
I_q\bigr).$$ 
By definition of the inductive limit topology, the $F_p$ induce a homeomorphism
 $F:C_q^B\to G_\infty/\!/K_\infty.$
Again we use the agreement that for all integers $p$ and for $p=\infty$,
$K_p$-biinvariant continuous functions on
$G_p$ will be identified with   continuous
 functions on $\Pi_q$. When doing so, the statement of Lemma 
\ref{posdef-p-infty} transfers to the present setting without changes.

\medskip

We now turn to the classification of spherical functions:

\begin{proposition}\label{class-spher-2} The spherical functions of $(G_p,K_p),$
considered as functions on  the chamber
$C_q^B,$
are precisely given by
the Dunkl-Bessel functions $\phi_\eta^{pd/2} (\xi):=J_{k(\mu,d)}^B(\xi, i\eta) $
with $\eta\in \mathbb C^q$, where $\mu = pd/2.$ 
Moreover, the positive definite  spherical functions of $(G_p,K_p)$ are given
by the   $\phi_{\eta}^{pd/2}$  with  $\eta\in \mathbb R^q$. 
\end{proposition}

\begin{proof}
The first statement is known from Dunkl theory,  see \cite{O}.
 The second then again follows
from Theorem 5.4 of \cite{W1}.
\end{proof}

 For the case $p=\infty$, we start with the following

\begin{lemma}\label{limit-product-formula2}
A continuous  $K_\infty$-biinvariant function $\phi:G_\infty\to \mathbb C$ 
is  Olshanski spherical 
if and only if the continuous function $\tilde\phi(\xi):=\phi(g_\xi)$ on $C_q^B$
 satisfies 
\begin{equation}\label{prod-formel-2}
\tilde\phi(\xi)\cdot \tilde\phi(\eta)
\,=\,  \int_{U_q}
\tilde\phi\bigl(\sigma(\sqrt{\underline \xi^2+ u\underline
\eta^2u^{-1}})\bigr)\>
du,
 \quad\quad \xi,\eta\in C_q^B.
\end{equation}
\end{lemma}

\begin{proof} Let $\phi$ be a continuous  $K_\infty$-biinvariant function on
  $G_\infty$ and $\tilde\phi$ defined as above.
  Then, by (\ref{def-ol-sph}) and the product formula for the spherical
  functions of $(G_p,K_p)$ (see e.g., p. 771 of \cite{R1}),
 $\phi$ is  Olshanski spherical iff $\tilde\phi$
  satisfies
\begin{align}\label{limit-eq3}
\tilde\phi(\xi)\cdot \tilde\phi(\eta)
\,=\,\lim_{p\to\infty} \frac{1}{\kappa_{pd/2}}
\int_{U_q}\int_{B_q}&\tilde\phi\bigl(\sigma(\sqrt{\underline \xi^2+u\underline
\eta^2u^*
+\underline \xi wu\underline \eta u^* +u\underline \eta u^* w^*\underline
\xi})\bigr)\notag \\
&
 \cdot \Delta(I-w^*w)^{pd/2-\gamma}\> dw\> du, \quad \xi, \eta \in C_q^B.
\end{align}
As
the probability measure
$$\,\kappa_{pd/2}^{-1}\cdot\Delta(I-w^*w)^{pd/2-\gamma} \,
dw\,$$ on $B_q$ tends weakly to the point
measure $\delta_0$ for $p\to \infty$, (\ref{limit-eq3}) is
equivalent to
the condition of the lemma.
\end{proof}

We now solve the functional equation (\ref{prod-formel-2}) by using the
Dunkl-Bessel functions $J_k^A$ of type A on the
Weyl chamber
 $$C_q^A:=\{\xi\in\mathbb R^q:\> \xi_1\ge\ldots\ge \xi_q\}\supset C_q^B.$$
For this we identify the space of double cosets of the
Gelfand pairs $(H_q\rtimes U_q, U_q)$ (where $U_q$ acts on $H_q$ by conjugation)
 with $C_q^A$ and recall from Section \ref{Dunkl-Bessel} that
the spherical functions of
$(H_q\rtimes U_q, U_q)$
are given by the functions  $x\mapsto J_{d/2}^A(\xi,\eta)$ with $\eta\in
\mathbb C^q$ by \cite{O}.

\begin{lemma}
A continuous function  $\tilde\phi$ on $C_q^B$ with $\tilde\phi(0)=1$
satisfies 
  (\ref{prod-formel-2}) if and only if there exists some  $b\in \mathbb C^q$
such that
$\tilde\phi(\xi)=J_{d/2}^A(\xi^2,b) $
 for all  $\xi\in C_q^B$. Here $\xi^2\in C_q^B $ means the vector which is
obtained
from
 $\xi$ by taking squares in each component.
\end{lemma}

\begin{proof} Let  $\tilde\phi\in C(C_q^B)$ with $\tilde \phi(0)=1$. Then
$\tilde\phi$ satisfies
(\ref{prod-formel-2}) if and only if the function $\psi(\xi):=\tilde\phi(\sqrt
\xi)$ on $C_q^B$ satisfies
\begin{equation}\label{prod-formel-3}
\psi(\xi)\cdot \psi(\eta)
\,=\,  \int_{U_q}
\psi(\sigma(\underline \xi+ u\underline \eta u^{-1}))\> du, 
 \quad\quad \xi, \eta\in C_q^B.
\end{equation}
In particular, for $a,b\in[0,\infty[$ 
    we have 
$$\psi((a,\ldots,a))\psi((b,\ldots,b))=\psi((a+b,\ldots,a+b)).$$
  As $\psi$ is  continuous with $\psi(0)=1,$ this implies that there exists some
$c\in \mathbb C$
such that 
\[\psi((a,\ldots,a))=e^{ca} \,\,\text{ for all }  a\in[0,\infty[.\] 
Precisely as in the proof or Lemma \ref{sol-conv-eq1}, it is now seen that
 $\psi$ can be uniquely extended from $C_q^B$ to a continuous function on
$C_q^A$ which
satisfies \eqref{prod-formel-3} for all $\xi,\eta\in
C_q^A$, namely by putting
\[\psi(\xi_1,\ldots,\xi_q):=\psi(\xi_1-\xi_q,\ldots, \xi_{q-1}-\xi_q,0)\cdot
e^{c\xi_q}
\quad\text{for }\, (\xi_1,\ldots,\xi_q)\in C_q^A.\]
Thus the extension  $\psi\in  C(C_q^A)$ is a spherical
function of the Gelfand pair $(H_q      \rtimes U_q,  U_q).$ On the other hand,
we know that the spherical 
functions of this Gelfand pair  are precisely 
  the
Dunkl-Bessel functions  $\xi\mapsto J_{d/2}^A(\xi,b)$ with $b\in 
\mathbb C^q$.
This proves the claim.
\end{proof}

The preceding lemmata yield the following characterization of the
 Olshanski spherical  functions.

\begin{theorem}\label{ident-ol2}
A continuous  $K_\infty$-biinvariant function $\phi: G_\infty\to \mathbb C$ is
Olshanski spherical 
 if and only if for some  $b\in\mathbb C^q$,
\[\phi(g_\xi)=J_{d/2}^A(\xi^2,-b) \, =: \psi_b(\xi)\,\,\, \text{ for all
}\,\xi\in C_q^B.\]
\end{theorem}

We next study  which  Olshanski spherical functions of $(G_\infty,K_\infty)$
appear as limits of spherical functions on $(G_p,K_p)$ for $p\to\infty$.
For this we employ the known convergence of  the Dunkl-Bessel functions of
type B to those of type A from \cite{RV1}:

\begin{lemma}\label{konv-matrix-bessel2}
For each  $b\in\mathbb C^q $, 
$$\lim_{\mu\to\infty}J_{k(\mu,d)}^B(2\sqrt{\mu} \xi, i
b)=J_{d/2}^A(\xi^2,-b^2)$$
 locally uniformly in $\xi\in C_q^B$. For $(\xi,b)\in C_q^B\times C_q^B,$ the 
convergence is even uniform.
\end{lemma}

We mention that this result also follows immediately
from Lemma \ref{konv-matrix-bessel} by taking  the means (\ref{Dunklchar}) and
 (\ref{int-rep-a-fkt}) for
the Dunkl-Bessel functions of type $B$ and A respectively. 
Lemma \ref{konv-matrix-bessel2} in combination with Proposition
\ref{class-spher-2} and Theorem
\ref{ident-ol2} implies the following result (again, we consider biinvariant
functions on 
$G_p$ and $G_\infty$ as functions on
 $C_q^B$):

\begin{theorem}\label{limit-spher-c3}
All Olshanski spherical functions $\psi_b\,,\, b \in \mathbb C^q$ of
$(G_\infty,K_\infty)$
appear as locally uniform limits of  spherical functions of  $(G_p,K_p)$ for
$p\to\infty$.

Moreover, those Olshanski spherical functions $\psi_b$ with $b\in
C_q^B$
appear  even as
uniform limits of
the  spherical functions of  $(G_p,K_p)$ for $p\to\infty$. 
\end{theorem}

We finally turn to the question which of the Olshanski spherical functions
$\psi_b$
are positive definite. The same argument as in the second part 
of the proof of Theorem \ref{pos-def1} immediately implies:

\begin{proposition}  The functions  $\psi_b$  with $b\in  C_q^B$ are
 positive definite Olshanski spherical functions of $(G_\infty,K_\infty)$.
\end{proposition}

We conjecture that  the converse statement holds as well. 
For $q=1$ we are   in fact
in the situation of Theorem \ref{pos-def1}, and the conjecture is clear.

For $q\ge2$ however,  there exist 
 spherical functions $\psi_b$  with $b\in  \mathbb C^q\setminus\mathbb R^q$
which are
$\mathbb R$-valued and bounded, and at present, we are unable to decide whether
these
functions are 
positive definite on $G_\infty$.

We give an example for $q=2$ and $\mathbb F=\mathbb R$:
Let $b=(i,-i)$ and $\xi=(\xi_1,\xi_2)\in C_2^B$, i.e. $\,\xi_1\ge \xi_2\ge0$ .
Then in view of
 integral representation  (\ref{int-rep-a-fkt}),
\begin{align}
\psi_b(\xi)&=\int_{O(2)} \exp\Bigl(-i \cdot tr\bigl(\underline \xi^2 \cdot u 
\begin{pmatrix} 1&0\\0&-1\end{pmatrix} u^{-1}\bigr)\Bigr) du\notag\\&
=\,\frac{1}{2\pi}\int_{-\pi}^{\pi} 
 \exp\bigl( -i(\xi_1^2-\xi_2^2)(\cos^2 t-\sin^2 t)\bigr) dt
\notag\\
& =\,\frac{1}{2\pi}\int_{-\pi}^{\pi} 
 \cos\left((\xi_1^2-\xi_2^2)(\cos^2 t-\sin^2 t)\right)dt.
\notag
\end{align}
It is unclear to us whether this function, which is a Bessel function, is
positive definite on
$G_\infty$, i.e., whether according to the version of Lemma \ref{posdef-p-infty}
in the
context of this section, this function is positive definite on $G_p$ for all
$p\ge q$.

\section{Positive integral representations of  matrix Bessel functions}

Consider integers $p_2\ge p_1\ge q$ and the associated indices $\mu_k:=p_kd/2$
($k=1,2$) of the matrix Bessel functions. Then  the functions
$\phi_s^{\mu_2}$ with $s\in\Pi_q$ as introduced in \eqref{def-phi-s}  represent
positive definite
biinvariant functions on $G_{p_1}$, and thus
by  Lemma \ref{posdef-p-infty}, positive definite biinvariant
functions on $G_{p_2}$. Therefore, by Bochner's theorem for hypergroups (see
\cite{J}), which
may be applied to  the associated matrix Bessel
hypergroup on $\Pi_q$ of index $\mu_1$, the function $\phi_s^{\mu_2}$ has a
representation
\begin{equation}\label{bochner_1}\phi_s^{\mu_2}(x)=\int_{\Pi_q}
\phi_t^{\mu_1}(x)\> d\nu_{p_1,p_2;s}(t)
\quad\quad(x\in\Pi_q)\end{equation}
with some unique probability measure $\nu_{p_1,p_2;s}\in M^1(\Pi_q)$.

In this section we shall determine these measures explicitely in two different
ways. Comparison of these results will then
lead to a projection result for multivariate beta distributions.
Our first approach was already carried out in \cite{H} for $\mathbb F=\mathbb
R$; it
relies on the following Laplace transform identity for the 
Bessel functions $ J_\mu$ which holds for general $\mathbb F$; see Proposition
XV.2.1 and Corollary VII.1.3
of \cite{FK}:

\begin{proposition}\label{laplace-trafo}
For  all $\mu\in\mathbb C$ with $\Re\mu>d(q-1)/2$ and $y\in\Omega_q$,
$$\int_{\Pi_q} J_\mu(x)e^{-\langle x,y\rangle} \Delta(x)^{\mu-n/q}\> dx 
= \Gamma_\Omega^q(\mu)\cdot \Delta(y)^{-\mu} \cdot e^{-tr(y^{-1})}$$
with the  Gamma function
$$\Gamma_\Omega^q(\mu)=\int_{\Pi_q}e^{-tr(x)} \Delta(x)^{\mu-n/q}\> dx =
(2\pi)^{(n-q)/2}
\Gamma(\mu)\Gamma(\mu-d/2)\cdots\Gamma(\mu-(q-1)d/2).$$
\end{proposition}

By the transformation formula for linear maps, we have for $m\in \Pi_q$
$$\int_{\Pi_q} f(mxm)\> dx=\Delta(m)^{2n/q}\int_{\Pi_q} f(x)\> dx.$$
 We thus obtain for
 $\mu\in\mathbb C$ with $\Re\mu>d(q-1)/2$, $m\in\Pi_q$ and $y\in\Omega_q$, that
\begin{equation}\label{laplace-trafo-modi}
\int_{\Pi_q} J_\mu(xm)e^{-\langle x,y\rangle} \Delta(x)^{\mu-n/q}\> dx 
= \Gamma_\Omega^q(\mu)\cdot \Delta(y)^{-\mu} \cdot e^{-tr(my^{-1})};
\end{equation}
c.f. equation (2.5) of \cite{H} for $\mathbb F=\mathbb R$, where a minus sign in
the
exponential is missing.
We notice that (\ref{laplace-trafo-modi}) may be also interpreted as 
a well-known formula for the Hankel transforms of Wishart distributions; see
for instance \cite{FK} or Section 5 of \cite{V1}.
Using injectivity and the convolution theorem for the Laplace transform,
 we deduce from (\ref{laplace-trafo-modi}) the following addition theorem:
\begin{proposition}\label{addition}
For all $\mu,\nu\in\mathbb C$ with $\Re\mu>d(q-1)/2$ and $\Re\nu>d(q-1)/2$ and
all $m_1,m_2\in\Pi_q$,
\begin{align} &J_{\mu+\nu}(x(m_1+m_2))\cdot\Delta(x)^{\mu+\nu-n/q}=\\&=
\frac{\Gamma_\Omega^q(\mu+\nu)}{\Gamma_\Omega^q(\mu)\Gamma_\Omega^q(\nu)}
\int_{\{y\in\Pi_q:\> y\le x\}} J_\mu(ym_1) \Delta(y)^{\mu-n/q} J_\nu((x-y)m_2)
\Delta(x-y)^{\nu-n/q} \>dy.
\notag\end{align}
\end{proposition}

Taking $m_1=m$, $m_2=0$,  $x=I_q$ (the identity matrix), and defining
$$\Pi_q^I:=\{y\in\Pi_q:\> y\le I_q\},$$
 we obtain the following Sonine-type integral representation of  
$J_{\mu+\nu}$ in terms of  $J_{\mu},$ which was for $\mathbb F= \mathbb R$
already proven in \cite{H}:

\begin{corollary}\label{int-rep-allg}
For all $\mu,\nu\in\mathbb C$ with $\Re\mu,\Re\nu >d(q-1)/2$  and $m\in\Pi_q$,
$$J_{\mu+\nu}(m)=\frac{\Gamma_\Omega^q(\mu+\nu)}{
\Gamma_\Omega^q(\mu)\Gamma_\Omega^q(\nu)}
\int_{ \Pi_q^I } J_\mu(ym) \Delta(y)^{\mu-n/q}\Delta(I_q-y)^{\nu-n/q} \>dy.$$
\end{corollary}

Notice that for $m=0$ this formula implies the known beta integral
\begin{equation}\label{beta}
\int_{\Pi_q^I }  \Delta(y)^{\mu-n/q}\Delta(I_q-y)^{\nu-n/q} \>dy
=\frac{\Gamma_\Omega^q(\mu)\Gamma_\Omega^q(\nu)}{\Gamma_\Omega^q(\mu+\nu)}
=:B_\Omega^q(\mu,\nu).
\end{equation}
Using these notions, we  define the beta distributions
$d\beta_{q;\mu,\nu}\in M^1(\Pi_q^I )$ on $\Pi_q^I$ for  real parameters
$\mu,\nu> d(q-1)/2$ by
 $$d\beta_{q;\mu,\nu}(y):=
\frac{1}{ B_\Omega^q(\mu,\nu)}  \Delta(y)^{\mu-n/q}\Delta(I_q-y)^{\nu-n/q}
\>dy$$
with $n=n(q,\mathbb F)$ as in Section 2.
Using Lemma \ref{properties-phi}(1), we obtain the following explicit form of
relation \eqref{bochner_1}:

\begin{corollary}\label{int-rep-allg-phi}
For all $\mu,\nu\in\mathbb C$ with $\Re\mu,\Re\nu>d(q-1)/2$ and $s,x\in\Pi_q$,
$$\phi_s^{\mu+\nu}(x)=
\int_{\Pi_q^I } \phi_{\sqrt{sys}}^\mu(x)\> d\beta_{q;\mu,\nu}(y).$$
\end{corollary}

\begin{remarks}\begin{enumerate}\itemsep=-1pt
 \item[\rm{(1)}] We conjecture that for all real parameters  $\nu\ge 0$, $
\mu>d(q-1)/2$
   and all $s\in\Pi_q$
there is a probability measure
$m_{\mu,\nu,s}\in M^1(\Pi_q)$ with compact support such that 
$$\phi_s^{\mu+\nu}(x)=\int_{\Pi_q}\phi_y^{\mu}(x) \> dm_{\mu,\nu,s}(y)   \,\,
\text{ for all }\,x\in\Pi_q $$
\item[\rm{(2)}] One can easily combine Corollary \ref{int-rep-allg-phi} with
  Lemma \ref{konv-matrix-bessel} for $\nu\to\infty$. 
This yields an explicit integral representation of the Olshanski spherical
  functions $\psi_{b}$ of Section 4 for $b\in\Pi_q$
in terms of the $\phi_s^\mu$. This formula describes  $\psi_{b}$ as the
Hankel transform of order $\mu$ of some Wishart distribution and 
 is  well-known, see Proposition XV.2.1 of \cite{FK} and Lemma 5.4. of
\cite{V1}, 
which is adapted to our notation.
\end{enumerate}
\end{remarks}

Next, we derive
 Corollaries \ref{int-rep-allg} and \ref{int-rep-allg-phi} 
in a different, more geometric way as follows:
Consider integers $1\le q\le\tilde p$ and $p\ge 2\tilde p$.
 For  $p,\tilde p$  define the associated parameters 
$\mu:=pd/2$ and $\tilde \mu:=\tilde pd/2$ of the  Bessel functions on $\Pi_q$. 
>From the 
   fact that the $\phi_\lambda^\mu$ are spherical  functions we conclude that
for $\lambda,x\in\Pi_q$
$$\phi_\lambda^\mu(x)=\int_{U_p} \exp\Bigl( i\cdot\Re
tr\Bigl((x,0) \,u\! \begin{pmatrix} \lambda\\ 0\end{pmatrix}  
\Bigr)\Bigr)du$$
with $(x,0)\in M_{q,p}$ and 
$\begin{pmatrix} \lambda\\ 0\end{pmatrix} \in M_{p,q}$
(see also \cite{R1}).
Now let $w\in B_{\tilde p}$ be the upper left $\tilde p\times \tilde p$-block of
$u$, i.e.
$\,u=\begin{pmatrix}w&*\\ *&*\end{pmatrix}.$ In the following, 
 $C_1,C_2,C_3    $ are constants depending on $q,\tilde p,p,d$.
 We first conclude from Lemma 2.1 of \cite{R2} that for  $p\ge 2\tilde p,$
$$\phi_\lambda^\mu(x)=C_1\int_{B_{\tilde p}}
 \exp\Bigl( i\cdot\Re tr\Bigl((x,0)w\begin{pmatrix}\lambda \\ 0\end{pmatrix}
\Bigr)\Bigr)\cdot \Delta(I-w^*\!w)^{pd/2 -d(\tilde p-1/2)-1}dw$$
where now $(x,0)\in M_{q,\tilde p}$ and
 $\begin{pmatrix}\lambda \\ 0\end{pmatrix} \in M_{\tilde p,q}$.
 Further, integration with respect to polar coordinates on $M_{\tilde p}$ 
according to \cite{FT} (see also p.759 of \cite{R1}) gives
$$\int_{M_{\tilde p}} f(w)dw =\,C_2 \int_{\Omega_{\tilde p}}  \int_{U_{\tilde
p}} f(u\sqrt r) \cdot \Delta(r)^{d/2-1} du\,dr.$$
Therefore
\begin{align}\phi_\lambda^\mu(x)=\,C_3\int_{\Pi_{\tilde p}^I} \Bigl(
\int_{U_{\tilde p}} 
 \exp\Bigl( i\cdot&\Re tr\Bigl((x,0)u\sqrt r \begin{pmatrix}\lambda \\
0\end{pmatrix}
\Bigr)\Bigr) du\Bigr)\cdot \notag\\
& \cdot \Delta(r)^{d/2-1}\Delta(I-r)^{pd/2 -d(\tilde p-1/2)-1}dr.\notag
\end{align}
By the definition of $\phi_\lambda^{\tilde \mu}$, we obtain
$$\phi_\lambda^\mu(x)=\,C_3 \int_{\Pi_{\tilde p}^I}
\phi_{\lambda(r)}^{\tilde \mu}(x)\Delta(r)^{d/2-1}\Delta(I-r)^{pd/2 -d(\tilde
p-1/2)-1}dr$$
with $\lambda(r):= \sqrt{(\lambda,0) r\begin{pmatrix}\lambda \\
0\end{pmatrix}}\in \Pi_q$.

Putting $x=0$
 and using (\ref{beta}) for $\tilde p$ instead of  $q$, we finally obtain from
the definition of $n$
that
\begin{equation}\label{mod-int-rep}
\phi_\lambda^\mu(x)=\int_{\Pi_{\tilde p}^I}\phi_{\lambda(r)}^{\tilde \mu}(x)\>
 d\beta_{\tilde p;\tilde pd/2,d(p-\tilde p)/2}(r).
\end{equation}
 We now compare equation (\ref{mod-int-rep}) with Corollary
\ref{int-rep-allg-phi}
 for 
$\mu=\tilde pd/2, \, \nu=(p-\tilde p)d/2$ and arbitrary $x\in\Pi_q$.
By the injectivity of the Hankel transform of index $\tilde p$ on $\Pi_q$ and by
analytic continuation
with respect to $\mu$ and $\nu$, we shall obtain
 the following projection result for  multivariate beta
 distributions:

\begin{proposition}\label{proj-beta}
 For integers $\tilde p\ge q\ge 1,$ consider the projection
$$P_{\tilde p,q}:\Pi_{\tilde p}^I\to \Pi_q^I,\quad r=\begin{pmatrix} y&*\\ *&
*\end{pmatrix}\longmapsto y.$$
Then for all real-valued parameters $\mu,\nu> d(\tilde p-1)/2$, the measure
$\beta_{q;\mu,\nu}$ is 
the push-forward
of the measure $\beta_{\tilde p;\mu,\nu}$ under $P_{\tilde p,q},$
$$P_{\tilde p,q}(\beta_{\tilde p;\mu,\nu})=\beta_{q;\mu,\nu}.$$
\end{proposition}

\begin{proof} Assume first that $\mu=d\tilde p/2$ and $\nu=d(p-\tilde p)/2$
  for some integer $p\ge 2\tilde p$. In this case, the statement follows from
  the preceding arguments.

Now consider integers $r\ge \tilde p\ge q$ and  $p\ge 2r$. We
obtain 
\begin{align}
\beta_{q;dr/2,d(p-r)/2}&=P_{r,q}(\beta_{r;dr/2,d(p-r)/2})=P_{\tilde p,q}\circ
P_{r,\tilde p}(\beta_{r;dr/2,d(p-r)/2})\notag\\
&=P_{\tilde p,q}(\beta_{\tilde
  p;dr/2,d(p-r)/2}).\notag
\end{align}
This equality means that 
 for the parameters $\mu=dr/2$ and $\nu=d(p-r)/2$ with integers $p,r$ with
$r\ge\tilde p$ and $p\ge 2r$, we have for all bounded continuous functions $f\in
C_b( \Pi_q^I)$ the identity
\begin{equation}\label{allg-pro}\int_{\Pi_{\tilde p}^I} f(P(x))\> d\beta_{\tilde
p;\mu,\nu}(x)=
\int_{\Pi_{q}^I} f(y)\> d\beta_{q;\mu,\nu}(y).
\end{equation}
By  definition of the beta distributions,
 both sides of this identity are, for fixed $f$, analytic in the variables 
$\mu,\nu\in\mathbb C$ with $\Re\mu,\Re\nu >d(\tilde p-1)/2 $. Precisely as in
\cite{R1}, it is 
 easily checked that
both sides also satisfy the  exponential growth conditions of Carlson's Theorem
(p. 186
of \cite{T}). 
This yields that (\ref{allg-pro}) holds for general $\mu$ and $\nu$,  which
 finishes the proof.
\end{proof}

We mention that  Proposition \ref{proj-beta} 
can also be derived, at least for $\mathbb F=\mathbb R$ and
$\mu=p/2$ and $\nu=r/2$ with suitable integers $p,r,$ from the construction of
the multivariate beta distributions in statistics; see Section 10.2 of
\cite{Fa} and references cited there.
Let us sketch this approach: we consider $M_{p,\tilde p}\,$- and $M_{r,\tilde
p}\,$-valued
independent
random variables $X$ and $Y$ respectively such that all entries of $X$ and $Y$
are
i.i.d. standard-normal distributed. Then $S:=X^tX$, $T:=Y^tY$, and $S+T$ are
$\Pi_{\tilde p}$ -valued and Wishart-distributed. Now form the unique lower
triangular matrix $C$ with nonnegative entries satisfying $CC^t=S+T$. It is
well-known (Section 10.2 of
\cite{Fa} and references cited there) that $L:=C^{-1}S(C^t)^{-1}$ is a
$\Pi_{\tilde p}^I$ -valued variable with distribution $\beta_{\tilde
p;p/2,r/2}$. 

Now consider some integer $1\le q\le \tilde p$. For any matrix $A\in M_{\tilde
  p}$ we denote by $\tilde A$ its upper left $q\times q$-block. It is
now easily checked from the triagular structure of $C$ and $C^{-1}$ that
$$\tilde L=\widetilde{ C^{-1}}\tilde S \widetilde{ (C^t)^{-1}}= (\tilde
C)^{-1}\tilde S
((\tilde C)^t)^{-1}$$
with $\tilde C\tilde C^t=\tilde S+\tilde T$. This observation readily leads
to a further proof of Proposition \ref{proj-beta} for $\mu=p/2$ and
$\nu=r/2$. Again, application of Carlson's theorem implies the general result.

\end{document}